\documentclass{article}
\usepackage[metapost]{mfpic}
\opengraphsfile{Projections03}
\usepackage{amscd}
\usepackage{amsxtra}
\usepackage{latexsym}
\usepackage{dsfont}
\usepackage{amsmath,amsthm,amsfonts,amssymb}
\usepackage[metapost]{mfpic}
\usepackage[all]{xy}
\usepackage{cite,color,epigraph}

\lccode`\-=`\-
\newcommand{\al}{\alpha}

\newcommand{\la}{\lambda}
\newcommand{\de}{\delta}

\newcommand{\bx}{\bar x}

\newcommand{\iv}{^{-1} }

\newcommand {\R} {\mathbb R}
\newcommand {\N} {\mathbb N}

\newcommand {\B} {\mathbb B}
\newcommand {\Sp} {\mathbb S}

\newcommand {\bd} {{\rm bd}\,}

\newcommand {\Limsup} {\mathop{{\rm Lim\,sup}\,}}
\newcommand {\sd} {\partial}
\newcommand {\Int} {{\rm int}\,}

\newcommand{\TO}[1]{\stackrel{#1}{\to}}

\newcommand{\norm}[1]{\left\Vert#1\right\Vert}
\newcommand{\abs}[1]{\left\vert#1\right\vert}
\newcommand{\set}[1]{\left\{#1\right\}}

\newcommand{\ang}[1]{\left\langle #1 \right\rangle}
\newcounter{mycount}

\newtheorem{theorem}{Theorem}
\newtheorem{proposition}[theorem]{Proposition}
\newtheorem{lemma}[theorem]{Lemma}

\newtheorem{corollary.pr}{Corollary}

\theoremstyle{definition}
\newtheorem{definition}[theorem]{Definition}
\theoremstyle{remark}
\newtheorem{remark}[theorem]{Remark}
\newtheorem{example}[theorem]{Example}
\theoremstyle{plain}


\DeclareMathOperator{\prox}{prox}

\DeclareMathOperator{\con}{cone}

\DeclareMathOperator{\gr}{gph}

\def\smartqedtr{\def\qedtr{\ifmmode\triangle\else{\unskip\nobreak\hfil
\penalty50\hskip1em\null\nobreak\hfil$\triangle$
\parfillskip=0pt\finalhyphendemerits=0\endgraf}\fi}}
\smartqedtr  

\title{Regularity of collections of sets and convergence of inexact alternating projections}
\author{Alexander Y. Kruger\\
\small
\emph{Centre for Informatics and Applied Optimization}\\
\small
\emph{Faculty of Science and Technology,
Federation University Australia}\\
\small
\emph{POB 663, Ballarat, Vic, 3350, Australia}\\
\small
\emph{a.kruger@federation.edu.au}
\and Nguyen H. Thao\\
\small
\emph{Centre for Informatics and Applied Optimization}\\
\small
\emph{Faculty of Science and Technology,
Federation University Australia}\\
\small
\emph{POB 663, Ballarat, Vic, 3350, Australia}\\
\small
\emph{hieuthaonguyen@students.federation.edu.au; nhthao@ctu.edu.vn}
}
\date{Dedicated to the memory of Jean Jacques Moreau}

\begin{document}
\maketitle

\begin{abstract}
We study the usage of regularity properties of collections of sets in convergence analysis of alternating projection methods for solving feasibility problems.
Several equivalent characterizations of these properties are provided.
Two settings of inexact alternating projections are considered and the corresponding convergence estimates are established and discussed.
\end{abstract}

\medskip


\section{Introduction}
In this article we study the usage of regularity properties of collections of sets in convergence analysis of alternating projection methods for solving \emph{feasibility problems}, i.e., finding a point in the intersection of several sets.

Given a set $A$ and a point $x$ in a metric space, the (metric) \emph{projection} of $x$ on $A$ is defined as follows:
\begin{gather*}
P_{A}(x):= \left\{a\in A \mid d(x,a)=d(x,A)\right\},
\end{gather*}
where $d(x,A):= \inf_{a\in A}d(x,a)$ is the distance from $x$ to $A$.
If $A$ is a closed subset of a finite dimensional space, then $P_{A}(x)\ne\emptyset$.
If $A$ is a closed convex subset of a Euclidean space, then $P_{A}(x)$ is a singleton.

Given a collection $\{A,B\}$ of two subsets of a metric space, we can talk about \emph{alternating projections}.

\begin{definition}[Alternating projections]\label{AP}
$\{x_n\}$ is a sequence of alternating projections for $\{A,B\}$ if
\begin{gather*}
x_{2n+1}\in P_B(x_{2n})
\quad\mbox{and}\quad
x_{2n+2}\in P_A(x_{2n+1})
\quad(n=0,1,\ldots).
\end{gather*}
\end{definition}

Investigations of convergence of the alternating projections to a point in the intersection of closed sets in the setting of a Hilbert space, or more often a finite dimensional Euclidean space, have long history which can be traced back to von Neumann; see the historical comments in \cite{LewLukMal09,DruIofLew14,NolRon}.
In the convex case, the key convergence estimates were established by Bregman \cite{Bre65} and Bauschke \& Borwein \cite{BauBor93}.
In the nonconvex case, in the finite dimensional setting, linear convergence of the method was shown by Lewis et al. \cite[Theorem~5.16]{LewLukMal09} under the assumptions of the \emph{uniform regularity} of the collection $\{A,B\}$ and \emph{super-regularity} of one of the sets; see the definitions and characterizations of these properties in Section~\ref{Prel}.

Throughout this paper, we assume that $A$ and $B$ are closed.

\begin{theorem}
[Linear convergence of alternating projections]\label{LLM}
Let $X$ be a finite dimensional Euclidean space.
Suppose that
\begin{enumerate}
\item
$\{A,B\}$ is uniformly regular at $\bx\in A\cap B$, i.e.,
\begin{align}\label{barc}
\bar{c}&:=
\sup \left\{-\ang{u,v}\mid u\in \overline{N}_{A}(\bx)\cap\overline{\B},\; v\in\overline{N}_{B}(\bx)\cap\overline{\B}\right\}<1;
\end{align}
\item
$A$ is super-regular at $\bx$.
\end{enumerate}
Then, for any $c\in(\bar{c},1)$, a sequence of alternating projections for $\{A,B\}$ with initial point sufficiently close to $\bx$ converges to a point in $A\cap B$ with $R-$linear rate $\sqrt{c}$.
\end{theorem}

$\overline{N}_{A}(\bx)$ and $\overline{N}_{B}(\bx)$ in \eqref{barc} stand for the \emph{limiting normal cones} to the corresponding sets at $\bx$; see definition \eqref{lFr} below.

Observe that $-\ang{u,v}$ in \eqref{barc} can be interpreted as the cosine of the angle between the cones $\overline{N}_{A}(\bx)$ and $-\overline{N}_{B}(\bx)$.

The role of the regularity (transversality-like) property (i) of $\{A,B\}$ and convexity-like property (ii) of $A$ in the convergence proof is analysed in Drusvyatskiy et al. \cite{DruIofLew14} and Noll \& Rondepierre \cite{NolRon}.
It has well been recognized that the uniform regularity assumption is far from being necessary for the linear convergence of alternating projections.
For example, as observed in \cite{DruIofLew14}, it fails when the affine span of $A\cup B$ is not equal to the whole space.

The drawback of the uniform regularity property as defined by \eqref{barc} from the point of view of the alternating projections is the fact that it takes into account all (limiting) normals to each of the sets while in many situations (like the one in the above example) some normals are irrelevant to the idea of projections.

Recently, there have been several successful attempts to relax the discussed above uniform regularity property by restricting the set of involved (normal) directions to only those relevant for characterizing alternating projections.
All the newly introduced regularity properties still possess some uniformity in the sense that they take into account directions originated from points in a neighbourhood of the reference point and some estimate is required to hold uniformly over all such directions.

Bauschke et al. \cite{BauLukPhaWan13.1,BauLukPhaWan13.2} suggested restricting the set of normals participating in \eqref{barc} by replacing $\overline{N}_{A}(\bx)$ and $\overline{N}_{B}(\bx)$ with \emph{restricted limiting normal cones} $\overline{N}_{A}^B(\bx)$ and $\overline{N}_{B}^A(\bx)$ depending on both sets and attuned to the method of alternating projections.
For example, the cone $\overline{N}_{A}^B(\bx)$ consists of limits of sequences of the type $t_k(b_k-a_k)$ where $t_k>0$, $b_k\in B$, $a_k$ is a projection of $b_k$ on $A$, and $a_k\to\bx$; cf. definitions \eqref{Bprox} and \eqref{BlFr}.
Bauschke et al. also adjusted (weakened) the notion of super-regularity accordingly (by considering \emph{joint super-restricted regularity} taking into account the other set) and, under these weaker assumptions, arrived at the same conclusion as in Theorem~\ref{LLM}; cf. \cite[Theorem~3.14]{BauLukPhaWan13.1} and Theorem~\ref{BLPW-T} below.

The idea of Bauschke et al. has been further refined by Drusvyatskiy et al. \cite[Definition~4.4]{DruIofLew14} who observed that it is sufficient to consider only sequences $t_k(b_k-a_k)$ as above with $b_k\to\bx$; cf. Definition~\ref{BLPW2} below.
In this case, $a_k\to\bx$ automatically.

In \cite{DruIofLew14}, the authors suggested also another way of weakening the uniform regularity condition \eqref{barc}.
Instead of measuring the angles between (usual or restricted in some sense) normals (and negative normals) to the sets, they measure the angles between vectors of the type $a-b$ with $a\in A$ and $b\in B$ and each of the cones $N_B^{\prox}(b)$ and $-N_A^{\prox}(a)$.
At least one of the angles must be sufficiently large when $a$ and $b$ are sufficiently close to $\bx$; cf. \cite[Definition~2.2]{DruIofLew14}.
Assuming this property and using a different technique, Drusvyatskiy et al. produced a significant advancement in convergence analysis of projection algorithms by establishing (see \cite[Corollary~4.2]{DruIofLew14}) $R-$linear convergence of alternating projections without the assumption of super-regularity of one of the sets (and with a slightly different convergence estimate).
The idea is closely related to the more general approach, where the feasibility problem is reformulated as a problem of minimizing a coupling function, and the property introduced in \cite{DruIofLew14} is sufficient for the coupling function to satisfy the Kurdyka-{\L}ojasiewicz inequality \cite[Proposition 4.1]{AttBolRedSou10}.

The two relaxed regularity properties introduced in \cite{DruIofLew14} are in general independent; cf. Examples~\ref{inh_not_int} and \ref{int_not_inh}.

The next step has been made by Noll and Rondepierre \cite[Definition~1]{NolRon}.
They noticed that, when dealing with alternating projections, the main building block of the method consists of two successive projections:
\begin{equation}\label{a_b}
a_1\in A,\quad b\in P_B(a_1)\quad\mbox{and}\quad a_2\in P_A(b)
\end{equation}
and it is sufficient to consider only the (proximal) normal directions determined by $a_1-b$ and $b-a_2$ for all $a_1,b,a_2$ in a neighbourhood of the reference point satisfying \eqref{a_b}.
In fact, in \cite{NolRon}, a more general setting is studied which allows for nonlinear convergence estimates under more subtle nonlinear regularity assumptions.


Another important advancement in this area is considering in \cite{LewLukMal09} of \emph{inexact} alternating projections.
Arguing that finding an exact projection of a point on a closed set is in general a difficult problem by itself, Lewis et al. relaxed the requirements to the sequence $\{x_n\}$ in Definition~\ref{AP} by allowing the points belonging to one of the sets to be ``almost'' projections.
Assuming that the other set is super-regular at the reference point, they established in \cite[Theorem~6.1]{LewLukMal09} an inexact version of Theorem~\ref{LLM}.

In the next section, we discuss and compare the uniform regularity property of collections of sets and its relaxations mentioned above.
Several equivalent characterizations of these properties are provided in a uniform way simplifying the comparison.

The terminology employed in \cite{LewLukMal09, BauLukPhaWan13.2,BauLukPhaWan13.1, DruIofLew14,NolRon} for various regularity properties is not always consistent.
We have not found a better way of handling the situation, but to use the terms \emph{BLPW-restricted regularity}, \emph{DIL-res\-tricted regularity}, and \emph{NR-restricted regularity} for the properties introduced in Bauschke, Luke, Phan, and Wang \cite{BauLukPhaWan13.2,BauLukPhaWan13.1}, Drusvyatskiy, Ioffe and Lewis \cite{DruIofLew14}, and Noll and Rondepierre \cite{NolRon}, respectively.
The refined version of \emph{BLPW-restricted regularity} due to Drusvyatskiy et al. \cite{DruIofLew14} is referred to in this article as \emph{BLPW-DIL-restricted regularity}.

In Section~\ref{S3}, we study two settings of inexact alternating projections under the assumptions of DIL-restricted regularity and uniform regularity, respectively, and establish and discuss the corresponding convergence estimates.

Our basic notation is standard; cf. \cite{Mor06.1,RocWet98, DonRoc09}.
For a normed linear space $X$, its topological dual is denoted $X^*$ while $\langle\cdot,\cdot\rangle$ denotes the bilinear form defining the pairing between the two spaces.
If $X$ is a Hilbert space, $X^*$ is identified with $X$ while $\langle\cdot,\cdot\rangle$ denotes the scalar product.
If $\dim X<\infty$, then $X$ is usually assumed equipped with the Euclidean norm.
The open and closed unit balls and the unit sphere in a normed space are denoted $\B$, $\overline{\B}$ and $\mathbb{S}$, respectively.
$\B_\delta(x)$
stands for the open
ball with radius $\delta>0$ and center $x$.
We use the convention $\B_0(x)=\{x\}$.

\section{Uniform regularity and related regularity properties}\label{Prel}

In this section, we discuss and compare the uniform regularity property of collections of sets and its several relaxations which are used in convergence analysis of projection methods.

\subsection{Uniform regularity}
The uniform regularity property has been studied in \cite{Kru06.1,Kru09.1,KruLop12.1,KruTha13.1,KruTha15}.
Below we consider the case of a collection $\{A,B\}$ of two nonempty closed subsets of a normed linear space.

\begin{definition}\label{UR}
Suppose $X$ is a normed linear space.
The collection $\{A,B\}$ is uniformly regular at $\bx\in A\cap B$ if there exist positive numbers $\alpha$ and $\delta$ such that
\begin{equation*}
(A-a-x) \bigcap(B-b-y) \bigcap(\rho\B )\neq \emptyset
\end{equation*}
for all $\rho \in (0,\delta)$, $a\in A\cap \B_{\delta}(\bx)$, $b\in B\cap \B_{\delta}(\bx)$, and $x,y\in (\alpha\rho)\B$.
\end{definition}

The supremum of all $\alpha$ in Definition~\ref{UR} is denoted $\hat{\theta}[A,B](\bx)$ and provides a quantitative characterization of the uniformly regularity property, the latter one being equivalent to the inequality $\hat\theta[A,B](\bx)>0$.
It is easy to check from the definition that
\begin{gather}\label{theta}
\hat\theta[A,B](\bx)= \liminf_{\substack{a\TO{A}\bx,b\TO{B}\bx, \rho\downarrow0}} \dfrac{\theta_{\rho}[A-a,B-b](0)}{\rho},
\end{gather}
where
\begin{gather*}
\theta_{\rho}[A,B](\bx):=\sup\left\{r\ge 0\mid (A-x)\bigcap (B-y) \bigcap \B_\rho(\bx)\neq\emptyset,\; \forall x,y\in r\B \right\}
\end{gather*}
and $a\TO{A}\bx$ means that $a\to\bx$ with $a\in A$.

The next proposition contains several characterizations of the uniform regularity property from  \cite{Kru06.1,Kru09.1,KruLop12.1,KruTha13.1, KruTha15}.
In its parts (ii) and (iii), $N_A(a)$ stands for the \emph{Fr\'echet normal cone} to $A$  at $a\in A$:
\begin{gather}\label{Fr}
N_{A}(a) := \left\{u\in X^* \mid \limsup_{x\TO{A}a} \frac {\langle u,x-a \rangle}{\|x-a\|} \leq 0 \right\}.
\end{gather}

\begin{proposition}\label{P3}
Let $A$ and $B$ be closed subsets of $X$.
\begin{enumerate}
\item
Suppose $X$ is a normed linear space.\\
\underline{Metric characterization}:
\begin{gather}\label{P3-1}
\hat{\theta}[A,B](\bx)= \liminf_{\substack{z\to\bx;\,x,y\to 0\\z\notin (A-x)\cap(B-y)}}\frac{\max\left\{d(z,A-x),d(z,B-y)\right\}} {d\left(z,(A-x)\bigcap (B-y)\right)}.
\end{gather}
$\{A,B\}$ is uniformly regular at $\bx$ if and only if there exist positive numbers $\alpha$ and $\delta$ such that
\begin{gather}\label{P3-2}
\alpha d\left(z,(A-x)\bigcap (B-y)\right)\le\max\left\{d(z,A-x),d(z,B-y)\right\}
\end{gather}
for all $z\in \B_{\delta}(\bx)$ and $x,y\in\de\B$.
\item
Suppose $X$ is an Asplund space.\\
\underline{Dual characterization}.
\begin{multline}\label{P3-3}
\hat{\theta}[A,B](\bx)= \lim_{\rho\downarrow 0}\inf \Big\{\norm{u+v}\mid u\in N_{A}(a),\;v\in N_{B}(b),\\
a\in A\cap \B_{\rho}(\bx),\;b\in B\cap \B_{\rho}(\bx),\; \norm{u}+\norm{v}=1\Big\}.
\end{multline}
$\{A,B\}$ is uniformly regular at $\bx$ if and only if there exist positive numbers $\alpha$ and $\delta$ such that
\begin{equation}\label{P3-4}
\alpha\left(\|u\|+\|v\|\right) \le \|u+v\|
\end{equation}
for all $a\in A\cap \B_{\de}(\bx)$, $b\in B\cap \B_{\de}(\bx)$, $u\in N_{A}(a)$, and $v\in N_{B}(b)$.
\item
Suppose $X$ is a Hilbert space.\\
\underline{Angle characterization}.
If either $\bx\in\bd A\cap\bd B$ or $\bx\in\Int(A\cap B)$, then
\begin{equation}\label{P3-5}
\hat{\theta}^2[A,B](\bx)=\frac{1}{2}\left(1-\hat c[A,B](\bx)\right),
\end{equation}
where
\begin{multline}\label{P3-6}
\hat c[A,B](\bx):=\lim_{\rho\downarrow0}\sup\Big\{-\langle u,v\rangle\mid u\in N_A(a)\cap\Sp,\;v\in N_B(b)\cap\Sp,\\
a\in A\cap \B_{\rho}(\bx),\;b\in B\cap \B_{\rho}(\bx)\Big\}.
\end{multline}
Otherwise, $\hat{\theta}[A,B](\bx)=1$ and $\hat c[A,B](\bx)=-\infty$.

$\{A,B\}$ is uniformly regular at $\bx$ if and only if ${\hat c[A,B](\bx)<1}$, i.e., there exist numbers $\alpha<1$ and $\delta>0$ such that
\begin{equation}\label{P3-7}
-\langle u,v\rangle<\al
\end{equation}
for all $a\in A\cap \B_{\de}(\bx)$, $b\in B\cap \B_{\de}(\bx)$, $u\in N_{A}(a)\cap\Sp$, and $v\in N_{B}(b)\cap\Sp$.
\end{enumerate}
\end{proposition}

\begin{remark}
1. Regularity criteria \eqref{P3-2} and \eqref{P3-4} are formulated in terms of distances in the primal space and in terms of Fr\'echet normals, respectively.
This explains why we talk about, respectively, the metric and the dual characterizations in parts (i) and (ii) of Proposition~\ref{P3}.
The term ``angle characterization'' in part (iii) comes from the observation that $-\ang{u,v}$ in criterion \eqref{P3-7} can be interpreted as the cosine of the angle between the unit vectors $u$ and $-v$.

2. Constant \eqref{theta} is nonnegative while constant \eqref{P3-6} can take negative values.
It is easy to see from \eqref{P3-1} that $\hat\theta[A,B](\bx)\le1$ if $\bx\notin\Int(A\cap B)$ and $\hat\theta[A,B](\bx)=\infty$ otherwise.
Similarly, $\abs{\hat c[A,B](\bx)}\le1$ if $\bx\in\bd A\cap\bd B$ and $\hat c[A,B](\bx)=-\infty$ otherwise.

3. Unlike \cite{KruTha13.1}, we assume in \eqref{P3-1}, \eqref{P3-3}, and \eqref{P3-6} the standard conventions that the infimum and supremum of the empty set in $\R$ equal $+\infty$ and $-\infty$, respectively.
As a result, an additional assumption that either $\bx\in\bd A\cap\bd B$ or $\bx\in\Int(A\cap B)$ is needed in part (iii) to ensure equality \eqref{P3-5}.

4. Equality \eqref{P3-1} was proved in \cite[Theorem~1]{Kru05.1} while equality \eqref{P3-3} was established in \cite[Theorem~4(vi)]{Kru09.1}; see also \cite[Theorem~4]{Kru05.1} for a slightly weaker result containing inequality estimates.
Equality \eqref{P3-5} is a direct consequence of \cite[Theorem~2]{KruTha13.1}.
It can be also easily checked directly.
\end{remark}

If $\dim X<\infty$, then representations \eqref{P3-3} and \eqref{P3-6} as well as the corresponding criteria in parts (ii) and (iii) of Proposition~\ref{P3} can be simplified by using the limiting version of the Fr\'echet normal cones \eqref{Fr}.
If, additionally, $X$ is a Euclidean space, then one can also make use of proximal normals.

Recall (cf., e.g., \cite{Mor06.1}) that, in a Euclidean space, the \emph{limiting (Fr\'echet) normal cone} to $A$ at $\bx$ and the \emph{proximal normal cone} to $A$ at $a\in A$ are defined as follows:
\begin{gather}\label{lFr}
\overline{N}_{A}(\bx) :=\Limsup_{a\TO{A}\bar x} N_{A}(a)=\left\{x^*=\lim x_k^*\mid x_k^*\in N_{A}(a_k),\; a_k\TO{A}\bar x\right\},
\\\label{prox}
N_A^{\prox}(a):=\con\left(P_A^{-1}(a)-a\right)= \set{\la(x-a)\mid\la\ge0,\;a\in P_A(x)}.
\end{gather}
Their usage is justified by the following simple observations:
\begin{gather}\label{prox2}
N_A^{\prox}(a)\subset N_{A}(a)
\quad\mbox{and}\quad
\overline{N}_{A}(\bx)=\Limsup_{a\TO{A}\bar x} N_{A}^{\prox}(a).
\end{gather}

\begin{proposition}\label{P4}
Let $A$ and $B$ be closed subsets of $X$.
\begin{enumerate}
\item
Suppose $\dim X<\infty$.\\
\underline{Dual characterizations}.
\begin{align*}
\hat{\theta}[A,B](\bx)&=
\inf \left\{\norm{u+v}\mid u\in \overline{N}_{A}(\bx),\; v\in\overline{N}_{B}(\bx),\;
\norm{u}+\norm{v}=1\right\}
\\
&=\lim_{\rho\downarrow 0}\inf \Big\{\norm{u+v}\mid u\in N_{A}^{\prox}(a),\;v\in N_{B}^{\prox}(b),\\
&\qquad\qquad a\in A\cap \B_{\rho}(\bx),\;b\in B\cap \B_{\rho}(\bx),\; \norm{u}+\norm{v}=1\Big\}.
\end{align*}
$\{A,B\}$ is uniformly regular at $\bx$ if and only if one of the following two equivalent conditions is satisfied:
\begin{enumerate}
\item
$\overline{N}_A(\bx)\cap \left(-\overline{N}_B(\bx)\right)=\{0\}$;
\item
there exist positive numbers $\alpha$ and $\delta$ such that inequality \eqref{P3-4} holds true
for all $a\in A\cap \B_{\de}(\bx)$, $b\in B\cap \B_{\de}(\bx)$, $u\in N^{\prox}_{A}(a)$, and $v\in N^{\prox}_{B}(b)$.
\end{enumerate}
\item
Suppose $X$ is a Euclidean space.\\
\underline{Angle characterizations}.
\begin{align}\label{P4-2}
\hat{c}[A,B](\bx)&=
\sup \left\{-\ang{u,v}\mid u\in \overline{N}_{A}(\bx)\cap\Sp,\; v\in\overline{N}_{B}(\bx)\cap\Sp\right\}
\\\notag
&=\lim_{\rho\downarrow 0}\sup \Big\{-\ang{u,v}\mid u\in N_{A}^{\prox}(a)\cap\Sp,\;v\in N_{B}^{\prox}(b)\cap\Sp,
\\\notag
&\qquad\qquad\qquad a\in A\cap \B_{\rho}(\bx),\;b\in B\cap \B_{\rho}(\bx)\Big\}.
\end{align}
$\{A,B\}$ is uniformly regular at $\bx$ if and only if one of the following two equivalent conditions is satisfied:
\begin{enumerate}
\item
$\set{(u,v)\in \left(\overline{N}_{A}(\bx)\cap\Sp\right)\times \left(\overline{N}_{B}(\bx)\cap\Sp\right)\mid \ang{u,v}=-1}=\emptyset$;
\item
there exist numbers $\alpha<1$ and $\delta>0$ such that inequality \eqref{P3-7} holds true
for all $a\in A\cap \B_{\de}(\bx)$, $b\in B\cap \B_{\de}(\bx)$, $u\in N_{A}^{\prox}(a)\cap\Sp$, and $v\in N_{B}^{\prox}(b)\cap\Sp$.
\end{enumerate}
\end{enumerate}
\end{proposition}

\begin{remark}
1. Condition (a) in part (i) of the above proposition is a ubiquitous qualification condition in optimization and variational analysis; cf. \emph{basic qualification condition} \cite{Mor06.1} and \emph{transversality} condition \cite{LewMal08,DruIofLew14}.

2. If one replaces $\Sp$ with $\overline{\B}$ in representation \eqref{P4-2}, one will get nonnegative constant \eqref{barc}.
The relationship between the two constants is straightforward: $\bar c=\max\{\hat c[A,B](\bx),0\}$.
\end{remark}

\subsection{Super-regularity}
In the next several subsections, we follow \cite{LewLukMal09} and \cite{BauLukPhaWan13.2,BauLukPhaWan13.1}, respectively.
Although some definitions and assertions are valid in arbitrary Hilbert spaces, in accordance with the setting of \cite{LewLukMal09} and \cite{BauLukPhaWan13.2,BauLukPhaWan13.1}, we assume in these two subsections that
$X$ is a finite dimensional Euclidean space.

Unlike the uniform regularity, the super-regularity property is defined for a single set.
The next definition contains a list of equivalent characterizations of this property which come from \cite[Definition~4.3, Proposition~4.4, and Corollary~4.10]{LewLukMal09}, respectively.

\begin{definition}\label{sup_reg}
A closed subset $A\subset X$ is super-regular at a point $\bx\in A$ if one of the following equivalent conditions is satisfied:
\begin{enumerate}
\item
for any $\gamma>0$, there exists a $\de>0$ such that
$$
\langle x-x_A,a-x_A\rangle\le \gamma\|x-x_A\|\,\|a-x_A\|
$$
for all $x\in\B_\de(\bx)$, $x_A\in P_{A}(x)$, and $a\in A\cap\B_\de(\bx)$;
\item
for any $\gamma>0$, there exists a $\delta>0$ such that
\begin{gather}\label{l2-1}
\langle u,x-a\rangle\le \gamma\|u\|\,\|x-a\|
\end{gather}
for all $x,a\in A\cap\B_\de(\bx)$ and $u\in{N}_{A}(a)$;
\item
for any $\gamma>0$, there exists a $\delta>0$ such that
$$
\langle v-u,y-x\rangle\ge -\gamma\|y-x\|
$$
for all $x,y\in A\cap\B_\de(\bx)$ and $u\in{N}_{A}(x)$ and $v\in{N}_{A}(y)$.
\end{enumerate}
\end{definition}

\begin{remark}
1. Super-regu\-larity is a kind of local ``near convexity'' property, refining or complementing a number of properties of this kind: \emph{Clarke regularity} \cite{ClaSteWol95,RocWet98}, \emph{amenability} \cite{RocWet98}, \emph{prox-regularity} \cite{PolRocThi00,RocWet98}, and \emph{subsmoothness} \cite{AusDanThi05} (cf. \emph{first order Shapiro property} \cite{Sha94}).
For a detailed discussion and comparing of the properties we refer the reader to \cite{LewLukMal09}.

2. Super-regu\-larity of one of the sets is an important ingredient of the convergence analysis of projection methods following the scheme initiated in Lewis et al. \cite{LewLukMal09}; cf. Theorems~\ref{LLM} and \ref{T31}.
In fact, a weaker ``quantified'' version of this property corresponding to fixing $\gamma>0$ in Definition~\ref{sup_reg} (and Definition~\ref{Bsup_reg} below), i.e., a kind of \emph{$\gamma$-super-regu\-larity} is sufficient for this type of analysis; cf.  \cite[Definition~8.1]{BauLukPhaWan13.2} and \cite[Definition~2]{NolRon} (The latter definition introduces a more advanced H\"older version of this property.)
Of course for alternating projections to converge, $\gamma$ must be small and the convergence rate depends on $\gamma$.
\end{remark}

\subsection{Restricted normal cones and restricted super-regu\-larity}
There have been several successful attempts to relax the discussed above regularity properties by restricting the set of involved (normal) directions to only those relevant for characterizing alternating projections.

The definitions of \emph{restricted normal cones} to a set introduced in \cite{BauLukPhaWan13.2} take into account another set and generalize proximal and limiting normal cones \eqref{prox} and \eqref{lFr} in the setting of a Euclidean space:
\begin{gather}\label{Bprox}
N_A^{B-\prox}(a):=\con\left((P_A^{-1}(a)\cap B)-a\right),
\\\label{BlFr}
\overline{N}_A^B(\bx):= \Limsup_{a\TO{A}\bx}N_A^{B-\prox}(a).
\end{gather}
Sets \eqref{Bprox} and \eqref{BlFr} are called, respectively, the \emph{$B$-proximal normal cone} to $A$ at $a\in A$ and \emph{$B$-limiting normal cone} to $A$ at $\bx$.
When $B$ is the whole space, they obviously coincide with \eqref{prox} and \eqref{lFr} (cf. the representation of the limiting normal cone given by the equality in \eqref{prox2}).
Note that cones \eqref{Bprox} and \eqref{BlFr} can be empty.

Similarly to \eqref{Bprox}, one can define also the \emph{$B$-Fr\'echet normal cone} to $A$ at $a\in A$:
\begin{gather*}
N_A^{B}(a):=N_A(a)\cap\con\left(B-a\right)
\end{gather*}
and the corresponding limiting one.
The following inclusions are straightforward:
\begin{gather*}
N_A^{B-\prox}(a)\subset N_A^{B}(a)\subset N_A(a).
\end{gather*}

\begin{definition}\label{Bsup_reg}
A closed subset $A\subset X$ is $B$-super-regular at a point $\bx\in A$ if, for any $\gamma>0$, there exists a $\de>0$ such that condition \eqref{l2-1} holds true
for all $x,a\in A\cap\B_\de(\bx)$ and $u\in{N}_{A}^{B-\prox}(a)$.
\end{definition}

\begin{remark}
As observed in \cite{BauLukPhaWan13.2}, $B$-proximal normals in Definition~\ref{Bsup_reg} can be replaced with $B$-limiting ones.
Similarly, in Definition~\ref{sup_reg}(ii) and (iii), one can replace Fr\'echet normals with limiting ones.
\end{remark}

\subsection{BLPW-restricted regularity}
The next definition introduces a modification of the property used in the angle characterization of the uniform regularity in Proposition~\ref{P3}(iii). This new property and its subsequent characterizations and application in convergence estimate (Theorem~\ref{BLPW-T}) originate in Bauschke, Luke, Phan, and Wang \cite{BauLukPhaWan13.2,BauLukPhaWan13.1}.
We are going to use for the regularity property of a collection of two sets discussed below the term \emph{BLPW-restricted regularity}.

\begin{definition}\label{BLPW}
A collection of closed sets $\{A,B\}$ is BLPW-restrictedly regular at $\bx\in A\cap B$ if \begin{align}\notag
\hat{c}_{1}[A,B](\bx)
&:=\lim_{\rho\downarrow 0}\sup \Big\{-\ang{u,v}\mid u\in N_{A}^{B-\prox}(a)\cap\Sp,\;v\in N_{B}^{A-\prox}(b)\cap\Sp,
\\\label{BLPW-1}
&\qquad\qquad\qquad a\in A\cap \B_{\rho}(\bx),\;b\in B\cap \B_{\rho}(\bx)\Big\}<1,
\end{align}
i.e., there exist numbers $\alpha<1$ and $\delta>0$ such that condition \eqref{P3-7} holds
for all $a\in A\cap \B_{\de}(\bx)$, $b\in B\cap \B_{\de}(\bx)$, $u\in N_{A}^{B-\prox}(a)\cap\Sp$, and $v\in N_{B}^{A-\prox}(b)\cap\Sp$.
\end{definition}

\begin{proposition}\label{P14}
\begin{enumerate}
\item
The following representation holds true:
\begin{align}\label{P14-1}
\hat{c}_{1}[A,B](\bx)
&=\sup \left\{-\ang{u,v}\mid u\in \overline{N}_{A}^B(\bx)\cap\Sp,\; v\in\overline{N}_{B}^A(\bx)\cap\Sp\right\}.
\end{align}
\item
If either $\overline{N}_{A}^B(\bx)\cap\Sp\ne\emptyset$ and $\overline{N}_{B}^A(\bx)\cap\Sp\ne\emptyset$, or $\overline{N}_{A}^B(\bx)\cap\Sp= \overline{N}_{B}^A(\bx)\cap\Sp=\emptyset$, then $$\hat{c}_{1}[A,B](\bx)=1-2\hat{\theta}^2_1[A,B](\bx),$$
where
\begin{align*}
\hat{\theta}_1[A,B](\bx)
=\lim_{\rho\downarrow 0}\inf &\Big\{\norm{u+v}\mid u\in N_{A}^{B-\prox}(a),\;v\in N_{B}^{A-\prox}(b),\\
&a\in A\cap \B_{\rho}(\bx),\;b\in B\cap \B_{\rho}(\bx),\; \norm{u}+\norm{v}=1\Big\}
\\
=\inf& \left\{\norm{u+v}\mid u\in \overline{N}_{A}^{B}(a),\;v\in \overline{N}_{B}^{A}(b),\; \norm{u}+\norm{v}=1\right\}.
\end{align*}
\item
A collection of closed sets $\{A,B\}$ is BLPW-restrictedly regular at $\bx\in A\cap B$ if and only if one of the following conditions holds true: \begin{gather}\label{P14-2}
\bar{c}_1:=
\sup \left\{-\ang{u,v}\mid u\in \overline{N}_{A}^B(\bx)\cap\overline{\B},\; v\in\overline{N}_{B}^A(\bx)\cap\overline{\B}\right\}<1,
\\\notag
\hat{\theta}_1[A,B](\bx)>0,
\\\label{P14-3}
\overline{N}_A^B(\bx)\cap \left(-\overline{N}_B^A(\bx)\right)\subseteq\{0\}.
\end{gather}
\end{enumerate}
\end{proposition}

\begin{remark}
1. The difference between formula \eqref{P14-1} and definition of $\bar{c}_1$ in \eqref{P14-2} is that, in the latter one, closed unit balls are used instead of spheres.
As a result, $\bar{c}_1$ is either nonnegative or equal $-\infty$.
(The latter case is possible because restricted normal cones can be empty.)
At the same time, conditions $\hat{c}_{1}[A,B](\bx)<1$ and $\bar{c}_1<1$ are equivalent and $\bar{c}_1$ can be used for characterizing BLPW-restricted regularity.
The inequality $\bar{c}_1\le\bar{c}$, where $\bar{c}$ is given by \eqref{barc}, is obvious.
It can be strict; cf. \cite[Example~7.1]{BauLukPhaWan13.2}.

2. In \cite{BauLukPhaWan13.2}, a more general setting of four sets $A,B,\widetilde A,\widetilde B$ is considered with the $A$- and $B$-proximal and limiting normals cones in Definition~\ref{BLPW} and Proposition~\ref{P14} replaced by their $\widetilde A$ and $\widetilde B$ versions.
As described in \cite[Subsection~3.6]{BauLukPhaWan13.1}, this provides additional flexibility in applications when determining regularity properties.
To simplify the presentation, in this article we set $\widetilde A=A$ and $\widetilde B=B$.

3. Condition~\eqref{P14-3} is referred to in \cite{BauLukPhaWan13.2} as \emph{$(A,B)$-qualification condition} while constant \eqref{BLPW-1} is called the \emph{limiting CQ number}.
\end{remark}

\begin{theorem}\label{BLPW-T}
Let $X$ be a finite dimensional Euclidean space.
Suppose that
\begin{enumerate}
\item
$\{A,B\}$ is BLPW-restrictedly regular at $\bx\in A\cap B$;
\item
$A$ is $B$-super-regular at $\bx$.
\end{enumerate}
Then, for any $c\in(\bar{c}_{1},1)$, a sequence of alternating projections for $\{A,B\}$ with initial point sufficiently close to $\bx$ converges to a point in $A\cap B$ with $R-$linear rate $\sqrt{c}$.
\end{theorem}

\subsection{BLPW-DIL-restricted regularity}
The concept of BLPW-restricted regularity was further refined in Drusvyatskiy, Ioffe and Lewis \cite[Definition~4.4]{DruIofLew14}.
We are going to call the amended property \emph{BLPW-DIL-restricted regularity}.

\begin{definition}\label{BLPW2}
A collection of closed sets $\{A,B\}$ is BLPW-DIL-restrictedly regular at $\bx\in A\cap B$ if
\begin{align}\notag
\hat{c}_{2}[A,B](\bx)
&:=\lim_{\rho\downarrow 0}\sup \Big\{-\frac{\ang{a-b_a,b-a_b}} {\|a-b_a\|\,\|b-a_b\|}\mid b_a\in P_B(a),\;a_b\in P_A(b),
\\\label{BLPW2-1}
&\qquad\qquad\qquad a\in(A\setminus B)\cap \B_{\rho}(\bx),\;b\in(B\setminus A)\cap \B_{\rho}(\bx)\Big\}<1,
\end{align}
i.e., there exist numbers $\alpha<1$ and $\delta>0$ such that
$$
-\ang{a-b_a,b-a_b}<\al\|a-b_a\|\,\|b-a_b\|
$$
for all $a\in(A\setminus B)\cap \B_{\delta}(\bx)$, $b\in(B\setminus A)\cap \B_{\delta}(\bx)$, $b_a\in P_B(a)$, and $a_b\in P_A(b)$.
\end{definition}

\begin{remark}
The property in Definition~\ref{BLPW2} is referred to in \cite{DruIofLew14} as \emph{inherent transversality}.
\end{remark}

An analogue of Proposition~\ref{P14} holds true with constant $\hat{\theta}_1[A,B](\bx)$ replaced by
\begin{multline*}
\hat{\theta}_2[A,B](\bx)
=\frac{1}{2}\lim_{\rho\downarrow 0}\inf \Biggl\{\norm{\frac{a-b_a}{\|a-b_a\|}+\frac{b-a_b}{\|b-a_b\|}}\mid b_a\in P_B(a),\;a_b\in P_A(b),\\
a\in(A\setminus B)\cap \B_{\rho}(\bx),\;b\in(B\setminus A)\cap \B_{\rho}(\bx)\Biggr\}
\end{multline*}
and appropriate limiting objects.

It is easy to see that a BLPW-restrictedly regular collection is also BLPW-DIL-restrictedly regular, but the converse is not true in general.

\subsection{DIL-restricted regularity}
The next definition and its subsequent characterizations originate in Drusvyatskiy, Ioffe and Lewis \cite{DruIofLew14}.
We are going to use for the regularity property of a collection of two sets discussed below the term \emph{DIL-restricted regularity}.

Unlike \cite{DruIofLew14}, if not specified otherwise, we adopt in this subsection the setting of a general Hilbert space.

\begin{definition}\label{P19}
A collection of closed sets $\{A,B\}$ is DIL-restrictedly regular at $\bx\in A\cap B$ if
\begin{multline}\label{P19-1}
\hat{\theta}_{4}[A,B](\bx)
:=\lim_{\rho\downarrow 0}\inf_{\substack{a\in(A\setminus B)\cap \B_{\rho}(\bx),\,b\in(B\setminus A)\cap \B_{\rho}(\bx)}}
\\
\max\biggl\{d\left(\frac{b-a}{\|a-b\|}, N_{A}(a)\right),
d\left(\frac{a-b}{\|a-b\|}, N_{B}(b)\right)\biggr\}>0,
\end{multline}
i.e., there exist positive numbers $\gamma$ and $\delta>0$ such that
\begin{equation}\label{int_con}
\max\biggl\{d\left(\frac{b-a}{\|a-b\|}, N_{A}(a)\right),
d\left(\frac{a-b}{\|a-b\|}, N_{B}(b)\right)\biggr\}>\gamma
\end{equation}
for all $a\in(A\setminus B)\cap \B_{\delta}(\bx)$ and $b\in(B\setminus A)\cap \B_{\delta}(\bx)$.
\end{definition}

\begin{proposition}\label{DIL}
A collection of closed sets $\{A,B\}$ is DIL-restrictedly regular at $\bx\in A\cap B$ if and only if
\begin{multline}\label{DIL-1}
\hat{c}_{4}[A,B](\bx)
:=\lim_{\rho\downarrow 0}\sup \Big\{\frac{\min\{\ang{b-a,u},\ang{a-b,v}\}} {\|a-b\|}\mid u\in N_{A}(a)\cap\Sp,
\\
v\in N_{B}(b)\cap\Sp,\;
a\in(A\setminus B)\cap \B_{\rho}(\bx),\;b\in(B\setminus A)\cap \B_{\rho}(\bx)\Big\}<1,
\end{multline}
i.e., there exist numbers $\alpha<1$ and $\delta>0$ such that
$$
\min\{\ang{b-a,u},\ang{a-b,v}\}<\al\|a-b\|
$$
for all $a\in(A\setminus B)\cap \B_{\delta}(\bx)$, $b\in(B\setminus A)\cap \B_{\delta}(\bx)$, $u\in N_{A}(a)\cap\Sp$, and $v\in N_{B}(b)\cap\Sp$.

Moreover, $(\hat{c}_{4}[A,B](\bx))^2+(\hat{\theta}_{4}[A,B](\bx))^2=1$.
\end{proposition}

\begin{remark}
1. If $\dim X<\infty$, then, as usual, the Fr\'echet normals in \eqref{P19-1} and \eqref{DIL-1} can be replaced by the proximal ones:
\begin{multline*}
\hat{\theta}_{4}[A,B](\bx)
=\lim_{\rho\downarrow 0}\inf_{\substack{a\in(A\setminus B)\cap \B_{\rho}(\bx),\,b\in(B\setminus A)\cap \B_{\rho}(\bx)}}
\\
\max\biggl\{d\left(\frac{b-a}{\|a-b\|}, N_{A}^{\prox}(a)\right),
d\left(\frac{a-b}{\|a-b\|}, N_{B}^{\prox}(b)\right)\biggr\},
\end{multline*}
\begin{multline*}
\hat{c}_{4}[A,B](\bx)
=\lim_{\rho\downarrow 0}\sup \Big\{\frac{\min\{\ang{b-a,u},\ang{a-b,v}\}} {\|a-b\|}\mid u\in N_{A}^{\prox}(a)\cap\Sp,
\\
v\in N_{B}^{\prox}(b)\cap\Sp,\;
a\in(A\setminus B)\cap \B_{\rho}(\bx),\;b\in(B\setminus A)\cap \B_{\rho}(\bx)\Big\}.
\end{multline*}

2. In \cite{DruIofLew14}, the property in Definition~\ref{P19} is referred to as \emph{intrinsic
transversality}.
\end{remark}

The next two examples show that DIL-restricted regularity is in general independent of BLPW-DIL-restricted regularity.

\begin{example}
[BLPW-DIL-restricted regularity but not DIL-restricted re\-gularity; Figure~\ref{INH_VS_INT}]
\label{inh_not_int}
Define a function $f:[0,1]\to \R$ by
\begin{equation*}
f(t):=
\left\{
  \begin{array}{ll}
    0, & \mbox{if } t=0,\\
    -t+1/2^{n+1}, & \mbox{if } t\in (1/2^{n+1},3/2^{n+2}],\\
    t-1/2^n, & \mbox{if } t\in (3/2^{n+2},1/2^{n}],\; n=0,1,\ldots
  \end{array}
\right.
\end{equation*}
and consider the sets: $A=\gr f$ and $B=\{(t,t): t\in [0,1]\}$ and the point $\bx=(0,0)=A\cap B$ in $\R^2$.
Suppose $\R^2$ is equipped with the Euclidean norm.

\begin{figure}[!ht]
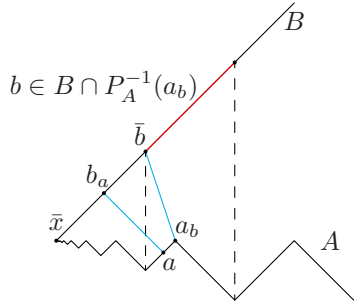

\begin{center}
\mfpic[45]{0}{2.5}{-.5}{2}
  \draw[black]\polyline{(2.5,-.5),(2,0),(1.5,-.5),(1,0),(.75,-.25),(.5,0),(.375,-.125),(.25,0),(.1875,-.0625),(.125,0),
  (.09375,-.03125),(.0625,0),(.046875,-.015625),(.03125,0),(.0234375,-.0078125),(.015625,0),(.01171875,-.00390625),(.0078125,0),(0,0),(2,2)}
  \dashed\polyline{(1.5,-.5),(1.5,1.5)}
  \dashed\polyline{(.75,-.25),(.75,.75)}
  \draw[cyan]\polyline{(.4,.4),(.9,-.1)}
  \draw[cyan]\polyline{(1,0),(.75,.75)}
  \draw[red]\polyline{(1.5,1.5),(.75,.75)}
  \tlabel[bc](.35,.45){$b_a$}
  \tlabel[bc](.7,.8){$\bar b$}
  \tlabel[cc](1.1,.1){$a_b$}
  \tlabel[tc](.95,-.15){$a$}
  \tlabel[cc](2.3,0){$A$}
  \tlabel[cc](2,1.85){$B$}
  \tlabel[cc](.4,1.3){$b\in B\cap P_A^{-1}(a_b)$}
  \tlabel[cc](0,.15){$\bx$}
  \pointcolor{black}
  \point[1.5pt]{(0,0),(.9,-.1),(1,0),(.75,.75),(.4,.4),(1.5,1.5)}
\endmfpic
\end{center}
\caption{BLPW-DIL-restricted regularity but not DIL-restricted regularity} \label{INH_VS_INT}
\end{figure}

It is easy to check that $f$ is a continuous function and consequently $A$ is closed; $f(1/2^n)=0$, $f(3/2^{n+2})=-1/2^{n+2}$, $n=0,1,\ldots$

Take any $a\in A\setminus B$, $b\in B\setminus A$, $b_a\in P_B(a)$, and $a_b\in P_A(b)$.
Thanks to the properties of the Euclidean distance, we have
\begin{gather*}
a_b=(1/2^n,0),\\
b\in B\cap P_A^{-1}(a_b)=\{(t,t)\mid t\in [3/2^{n+2},3/2^{n+1}]\},\\
a-b_a=k(1,-1)
\end{gather*}
for some $n\in\N$ and $k>0$.
Then,
\begin{align*}
\hat{c}_{2}[A,B](\bx)
&=\max_{b\in B\cap P_A^{-1}(a_b)} \set{\frac{\ang{(-1,1),b-a_b}} {\sqrt{2}\|b-a_b\|}}= \frac{1}{\sqrt{2}}\ang{(-1,1),\frac{\bar b-a_b}{\|\bar b-a_b\|}},
\end{align*}
where $\bar b:=(3/2^{n+2},3/2^{n+2})$, and consequently,
\begin{align*}
\hat{c}_{2}[A,B](\bx)
&=\frac{\ang{(-1,1),(-1,3)}}{\sqrt{2}\,\sqrt{10}} =\frac{2}{\sqrt{5}}>0.
\end{align*}
Hence, $\{A,B\}$ is BLPW-DIL-restrictedly regular at $\bx$.

Given an $n\in\N$, we choose $a:=(1/2^n,0)\in A\setminus B$ and $b:=(1/2^{n+1},1/2^{n+1}) \in B\setminus A$.
Then,
\begin{align*}
N_A^{\prox}(a)=&N_A(a)=\{(t_1,t_2):t_2\ge|t_1|\},\\
N_B^{\prox}(b)=&N_B(b)=\R(1,-1),
\end{align*}
and consequently,
\begin{equation*}
a-b=1/2^{n+1}(1,-1)\in N_B(b)\cap -N_A(a).
\end{equation*}
It follows that $\hat{c}_{4}[A,B](\bx)=1$ and $\{A,B\}$ is not DIL-restrictedly regular at $\bx$.
\qedtr\end{example}

\begin{example}
[DIL-restricted regularity but not BLPW-DIL-restricted re\-gularity; Figure~\ref{INT_VS_INH}]
\label{int_not_inh}
Consider two sets:
\begin{align*}
A=&\{(t,0): t\ge 0\}\cup \{(t,-t): t\ge 0\},\\
B=&\{(t,0): t\ge 0\}\cup \{(t,t): t\ge 0\}
\end{align*}
and the point $\bx=(0,0)\in A\cap B$ in $\R^2$.
Suppose $\R^2$ is equipped with the Euclidean norm.

\begin{figure}[!ht]
\begin{center}
\mfpic[45]{0}{2}{-1}{1}
  \draw[black]\polyline{(1,1),(0,0),(1.5,0)}
  \draw[black]\polyline{(0,0),(1,-1)}
  \draw[cyan]\polyline{(.7,.7),(.7,0)}
  \draw[cyan]\polyline{(.5,0),(.5,-.5)}
  \tlabel[tc](.5,-.6){$a$}
  \tlabel[cc](.7,-.1){$a_b$}
  \tlabel[cc](.5,.15){$b_a$}
  \tlabel[cc](.7,.85){$b$}
  \tlabel[cc](1,-.75){$A$}
  \tlabel[cc](1,.75){$B$}
  \tlabel[cc](1.2,.15){$A\cap B$}
  \tlabel[cc](0,.15){$\bx$}
  \pointcolor{black}
  \point[1.5pt]{(0,0),(.5,0),(.7,.7),(.5,-.5),(.7,0)}
\endmfpic
\end{center}
\caption{DIL-restricted regularity but not BLPW-DIL-restricted re\-gularity}\label{INT_VS_INH}
\end{figure}

For any $a=(t_1,-t_1)\in A\setminus B$ and $b=(t_2,t_2)\in B\setminus A$, we have
\begin{align*}
N_A^{\prox}(a)=&N_A(a)=\R(1,1),\\
N_B^{\prox}(b)=&N_B(b)=\R(1,-1).
\end{align*}
and consequently,
\begin{align*}
\hat{c}_{4}[A,B](\bx)
=\sup_{t_1>0,\,t_2>0} \frac{\min\{t_1,t_2\}} {\sqrt{t_1^2+t_2^2}}=\frac{1} {\sqrt{2}}<1.
\end{align*}
Hence, $\{A,B\}$ is DIL-restrictedly regular at $\bx$.

For any $a\in A\setminus B$, $b\in B\setminus A$, $b_a\in P_B(a)$, and $a_b\in P_A(b)$, we have
\begin{align*}
\frac{b-a_b} {\|b-a_b\|}= \frac{b_a-a} {\|a-b_a\|}.
\end{align*}
It follows that $\hat{c}_{2}[A,B](\bx)=1$ and $\{A,B\}$ is not BLPW-DIL-restrictedly regular at $\bx$.
\qedtr\end{example}

The next fact was established in
\cite[Proposition 4.5]{DruIofLew14}.

\begin{proposition}\label{inh_int_tra}
If $\dim X<\infty$, $\{A,B\}$ is BLPW-DIL-restrictedly regular at $\bx$, and both sets $A$ and $B$ are super-regular at $\bx$, then $\{A,B\}$ is DIL-restrictedly regular at $\bx$.
\end{proposition}

\begin{remark}
The assumption of super-regularity of both sets in Proposition \ref{inh_int_tra} is essential.
Indeed, in Example~\ref{inh_not_int}, $\{A,B\}$ is BLPW-DIL-res\-trictedly regular and $B$ is super-regular (in fact, convex), while $A$ is not and $\{A,B\}$ is not DIL-restrictedly regular.
\end{remark}

\subsection{NR-restricted regularity}
The next step in relaxing both BLPW- and DIL-restricted regularity properties while preserving the linear convergence of alternating projections has been done in Noll and Rondepierre \cite{NolRon}.
In what follows, the resulting property is called \emph{NR-restricted regularity}.

\begin{definition}\label{NR}
A collection of closed sets $\{A,B\}$ is NR-restrictedly regular at $\bx\in A\cap B$ if
\begin{multline*}
\hat{c}_{3}[A,B](\bx)
:=\lim_{\rho\downarrow 0}\sup \Big\{\frac{\ang{a_1-b,a_2-b}} {\|a_1-b\|\,\|a_2-b\|}\mid
a_1\in A,\,b\in P_B(a_1),\,a_2\in P_A(b)
\\
a_1,b,a_2\in \B_{\rho}(\bx),
\Big\}<1,
\end{multline*}
i.e., there exist numbers $\alpha<1$ and $\delta>0$ such that
$$
\ang{a_1-b,a_2-b}\le \al\|a_1-b\|\,\|a_2-b\|
$$
for all $a_1\in A\cap \B_{\delta}(\bx)$, $b\in P_B(a_1)\cap \B_{\delta}(\bx)$, and $a_2\in P_A(b)\cap \B_{\delta}(\bx)$.
\end{definition}

\begin{remark}
1. NR-restricted regularity property is not symmetric: NR-restric\-ted regularity of $\{A,B\}$ does not imply that $\{B,A\}$ is NR-restrictedly regular.

2. If $\{A,B\}$ is BLPW- or DIL-restrictedly regular at $\bx$, then it is NR-res\-trictedly regular at $\bx$ and the second implication can be strict \cite[Propositions~1 and 2 and Example 7.6]{NolRon}.
In fact, it is easy to check that NR-restricted regularity is implied by BLPW-DIL-restricted regularity.
Example~\ref{int_not_inh} shows that NR-res\-tricted regularity can be strictly weaker.

3. Theorem \ref{BLPW-T} remains valid if the assumption of BLPW-restricted regularity is replaced by that of NR-restricted regularity and $\bar{c}_1$ is replaced by $\hat{c}_{3}[A,B](\bx)$.

4. The property in Definition~\ref{NR} is referred to in \cite{NolRon} as \emph{separable intersection}.

5. In \cite{NolRon}, a more general H\"older-type property with exponent $\omega\in [0,2)$ is considered.
Definition \ref{NR} corresponds to that property with $\omega=0$.
For the convergence analysis, the authors of \cite{NolRon} introduce also a H\"older version of the superregularity property.
\end{remark}


\section{Convergence for inexact alternating projections}\label{S3}

In this section, we study two settings of inexact alternating projections under the assumptions of DIL-restricted regularity and uniform regularity, respectively, and establish the corresponding convergence estimates.

\subsection{Convergence for inexact alternating projections under DIL-restricted regularity}
Given a point $x$ and a set $A$ in a Hilbert space and numbers $\tau\in (0,1]$ and $\sigma\in[0,1)$, the \emph{$(\tau,\sigma)$-projection} of $x$ on $A$ is defined as follows:
\begin{gather}\label{ab}
P_{A}^{\tau,\sigma}(x):= \left\{a\in A \mid \tau \|x-a\|\le d(x,A),\; d\left(x-a,N_A(a)\right)\le \sigma\|x-a\|\right\}.
\end{gather}
One obviously has $P_{A}^{1,\sigma}(x)=P_{A}(x)$ for any $\sigma\in [0,1)$.
Observe also that the above definition requires $a$ to be an ``almost projection'' in terms of the distance $\|x-a\|$ being close to $d(x,A)$ and also $x-a$ being an ``almost normal'' to $A$.

\begin{definition}[Inexact alternating projections]\label{ine_alt_pro}
Given $\tau\in(0,1]$ and $\sigma\in[0,1)$,
$\{x_n\}$ is a sequence of $(\tau,\sigma)$-alternating projections for $\{A,B\}$ if
\begin{gather*}
x_{2n+1}\in P_B^{\tau,\sigma}(x_{2n})
\quad\mbox{and}\quad
x_{2n+2}\in P_A^{\tau,\sigma}(x_{2n+1})
\quad
(n=0,1,\ldots).
\end{gather*}
\end{definition}

The next statement is taken from \cite[Theorem~5.3]{DruIofLew14} where it is formulated in the setting of a finite dimensional Euclidean space.
It is a version of the general metric space \emph{Basic Lemma} from \cite{Iof00_}.
Recall that the (strong) \emph{slope} \cite{DMT} of $f$ at a point $u\in X$ with $f(u)<+\infty$ is defined as follows:
$$
|\nabla f|(u):=\limsup_{u'\to u,u'\neq u}\frac{f(u)-f(u')}{d(u',u)}.
$$

\begin{lemma}[Error bound]\label{new_err_bou}
Let $X$ be a complete metric space, $f:X\to\R\cup \{+\infty\}$ a lower semicontinuous  function, $x\in X$ with $f(x)<+\infty$, $\delta>0$, and $\alpha<f(x)$.
Suppose that
\begin{equation}\label{mu}
\mu:=\inf_{u\in\B_\delta(x),\,\alpha<f(u)\le f(x)}|\nabla f|(u)>\frac{f(x)-\alpha}{\delta}.
\end{equation}
Then
$S(f,\alpha):=\set{u\in X\mid f(u)\le\alpha}\neq \emptyset$
and
$$\mu d(x,S(f,\alpha))\le f(x)-\alpha.$$
\end{lemma}

If $X$ is an Asplund space, then a standard argument based on the \emph{subdifferential sum rule} (cf., e.g., \cite[Proposition~5(ii)]{FabHenKruOut10} or \cite[Proposition~6.9]{DruIofLew14}) shows that the primal space slopes in the definition of $\mu$ in \eqref{mu} can be replaced by the subdifferential slopes:
\begin{equation}\label{mu2}
\mu=\inf_{u\in\B_\delta(x),\,\alpha<f(u)\le f(x)}|\sd f|(u).
\end{equation}
Here
$$|\sd f|(u):=\inf_{x^*\in\sd f(u)}\|x^*\|_*,$$
where $\sd f(u)$ is the Fr\'echet subdifferential of $f$ at $u$ and $\|\cdot\|_*$ is the norm on $X^*$ dual to the norm on $X$ participating in the definition of the primal space slope.
Note that in general $|\nabla f|(u)\le|\sd f|(u)$.

If $X$ is a finite dimensional Euclidean space, then, instead of the Fr\'echet subdifferentials, one can use the \emph{proximal} subdifferentials $\sd^{\prox} f(u)$:
\begin{equation}\label{mu5}
\mu=\inf_{u\in\B_\delta(x),\,\alpha<f(u)\le f(x)}|\sd^{\prox} f|(u),
\end{equation}
where
$$|\sd^{\prox} f|(u):=\inf_{x^*\in\sd^{\prox} f(u)}\|x^*\|.$$

The next statement is a consequence of Lemma~\ref{new_err_bou}.
It extends slightly \cite[Theorem~3.1]{DruIofLew14}.
\begin{proposition}[Distance decrease]\label{dis_dec}
Let $A$ be a closed subset of a Hilbert space $X$, $a\in A$, $b\notin A$, $\delta>0$, and $\alpha<\|a-b\|$.
Suppose that
\begin{gather}\label{mu3}
\mu:=\inf_{\substack{u\in A\cap\B_\delta(a)\\\|u-b\|\le\|a-b\|}}d\left(\frac{b-u}{\|u-b\|},N_A(u)\right)>0.
\end{gather}
Then $d(b,A)\le\|a-b\|-\mu \delta$.

If $\dim X<\infty$, then the Fr\'echet normal cones $N_A(u)$ in \eqref{mu3} can be replaced by the proximal ones $N^{\prox}_A(u)$.
\end{proposition}
\begin{proof}
Consider the lower semicontinuous function $f=d(\cdot,b)+\iota_{A}$,
where $\iota_{A}$ is the \emph{indicator} function of $A$: $\iota_{A}(x)=0$ if $x\in A$ and $\iota_{A}(x)=+\infty$ if $x\notin A$.
Then $f(a)=\|a-b\|$ and
$$\sd f(u)=\frac{u-b}{\|u-b\|}+N_A(u),\quad
|\sd f|(u)=d\left(\frac{b-u}{\|u-b\|},N_A(u)\right)$$
for any $u\in A$.
It follows from the first part of Lemma~\ref{new_err_bou} and representation \eqref{mu2} that $d(b,A)\le\alpha$ for any
$\alpha \in (\|a-b\|-\mu \delta,\|a-b\|)$ and consequently, $d(b,A)\le\|a-b\|-\mu \delta$.

If $\dim X<\infty$, then instead of representation \eqref{mu2} one can use representation \eqref{mu5}.
\end{proof}

The next statement is essentially \cite[Lemma~3.2]{DruIofLew14}.
\begin{lemma}\label{lem3.2}
Any nonzero vectors $x$ and $y$ in a Hilbert space satisfy
\begin{equation*}
\norm{\frac{x}{\|x\|}-z}\le \frac{\|x-y\|}{\|y\|},
\end{equation*}
where $z:=\ang{\frac{x}{\|x\|},\frac{y}{\|y\|}}\frac{y}{\|y\|}$ is the projection of $\frac{x}{\|x\|}$ on $\R y$.
\end{lemma}

\begin{proof}
\begin{align*}
\left(\frac{\|x-y\|}{\|y\|}\right)^2-
\norm{\frac{x}{\|x\|}-z}^2
&=
\frac{\|x\|^2-2\langle x,y\rangle+\|y\|^2}{\|y\|^2}-
1+\ang{\frac{x}{\|x\|},\frac{y}{\|y\|}}^2
\\&=
\frac{1}{\|y\|^2}\left(\|x\|^2-2\langle x,y\rangle
+\frac{\langle x,y\rangle^2}{\|x\|^2}\right)
\\&=
\frac{1}{\|y\|^2}\left(\|x\|-\frac{\langle x,y\rangle}{\|x\|}\right)^2\ge0.
\end{align*}
\end{proof}

\begin{theorem}[Convergence of inexact alternating projections]\label{new_mai_res}
Suppose that $\{A,B\}$ is DIL-restrictedly regular at $\bx$,
$0\le\sigma<\hat{\theta}_{4}[A,B](\bx)$ and $0<\tau\le1$.
Then, for any $\gamma<\hat{\theta}_{4}[A,B](\bx)$ satisfying $0<\gamma-\sigma\le\tau$ and
$$c:=\tau\iv(1-\gamma^2+\gamma\sigma)<1,$$
any sequence of $(\tau,\sigma)$-alternating projections for $\{A,B\}$ with initial point sufficiently close to $\bx$ converges to a point in $A\cap B$ with $R-$linear rate $c$.
\end{theorem}
\begin{proof}
By Definition~\ref{P19}, there exists a $\rho>0$ such that condition \eqref{int_con} holds true
for all $a\in(A\setminus B)\cap \B_{\rho}(\bx)$ and $b\in(B\setminus A)\cap \B_{\rho}(\bx)$.

Let $a\in A\cap\B_{\rho'}(\bx)$ and $b\in P_B^{\tau,\sigma}(a)\cap\B_{\rho'}(\bx)$ where $\rho':=\rho/(1+2(\gamma-\sigma))$.
We are going to show that
$$
d(b,A)\le(1-\gamma^2+\gamma\sigma)\|b-a\|.
$$
If $b\in A$, the inequality holds true trivially.
Suppose $b\notin A$ and denote $\delta:=(\gamma-\sigma)\|b-a\|$.
Consider any point $u\in A\cap\B_{\delta}(a)$.
Since $\|u-a\|<(\gamma-\sigma)\|b-a\|\le\tau\|b-a\|\le d_B(a)$, we see that $u\notin B$; in particular, $a\notin B$ and $u\ne b$.
Let $z$ denote the projection of $\frac{u-b}{\|u-b\|}$ on $\R(a-b)$.
Then $\|z\|\le 1$ and, employing Lemma~\ref{lem3.2},
\begin{align*}
d\left(\frac{u-b}{\|u-b\|},N_B(b)\right)
&\le \left\|\frac{u-b}{\|u-b\|}-z\right\|+d(z,N_B(b))
\\
&\le \frac{\|u-a\|}{\|b-a\|}+ \sigma
<(\gamma-\sigma)+\sigma=\gamma.
\end{align*}
Since $\|u-\bx\|\le\|u-a\|+\|a-\bx\|<2(\gamma-\sigma)\rho'+\rho'=\rho$ and $\|b-\bx\|<\rho'<\rho$, we get from \eqref{int_con} that
$d\left(\frac{b-u}{\|u-b\|},N_A(u)\right)>\gamma$.
It follows from Proposition~\ref{dis_dec} that $d(b,A)\le\|a-b\|-\gamma\delta=(1-\gamma^2+\gamma\sigma)\|a-b\|$.
Hence,
\begin{align}\label{in}
\|a'-b\|\le\tau\iv d(b,A)\le c\|a-b\| \quad\mbox{for all}\quad a'\in P_A^{\tau,\sigma}(b).
\end{align}

Now we show that any sequence $\{x_n\}$ of $(\tau,\sigma)$-alternating projections for $\{A,B\}$ remains in $\B_{\rho'}(\bx)$ whenever $x_0\in\B_{\rho''}(\bx)$ where $\rho'':=\left(\frac{\tau\iv}{1-c}+1\right)\iv\rho'<\rho'$.
Indeed,
\begin{align*}
\|x_1-x_0\|\le
\tau\iv d(x_0,B)\le\tau\iv\|x_0-\bx\|.
\end{align*}
Let $n\in\N$ and $x_k\in\B_{\rho'}(\bx)$, $k=0,1,\ldots n$.
It follows from \eqref{in} that
\begin{align}\label{in2}
\|x_{k+1}-x_k\|\le c^k\|x_1-x_0\|\quad (k=0,1,\ldots n),
\end{align}
and consequently,
\begin{align*}
\|x_{n+1}-x_0\|&\le\sum_{k=0}^{n} c^k\|x_1-x_0\|\le\frac{1}{1-c}\|x_1-x_0\|,
\\
\|x_{n+1}-\bx\|&\le\left(\frac{\tau\iv}{1-c}+1\right)\|x_0-\bx\|<\rho'.
\end{align*}
Thanks to \eqref{in2}, $\{x_k\}$ is a Cauchy sequence containing two subsequences belonging to closed subsets $A$ and $B$, respectively.
Hence, it converges to a point in $A\cap B$ with $R-$linear rate $c$.
\end{proof}
\begin{remark}\label{R31}
1. When inexact alternating projections are close to being exact, i.e., $\tau$ and $\sigma$ are close to 1 and 0, respectively (cf. definition \eqref{ab}), then the assumptions of Theorem~\ref{new_mai_res} are easily satisfied (as long as $\hat{\theta}_{4}[A,B](\bx)>0$) while the convergence rate $c=\tau\iv(1-\gamma^2+\gamma\sigma)$ is mostly determined by the term $1-\gamma^2$.
Recall that $\gamma$ can be any number in $(0,\hat{\theta}_{4}[A,B](\bx))$.
Thanks to Proposition~\ref{P19}, $1-\gamma^2=(\gamma')^2$ where $\gamma'$ can be any number in $(\hat{c}_{4}[A,B](\bx),1)$.

2. When $\dim X<\infty$, the special case $\tau=1$ and $\sigma=0$ of Theorem~\ref{new_mai_res} recaptures \cite[Theorem~2.3]{DruIofLew14}.
The proof given above follows that of \cite[Theorem~2.3]{DruIofLew14}.

3. It can be of interest to consider a more advanced version of inexact alternating projections than the one given in Definition~\ref{ine_alt_pro}:
\begin{gather*}
x_{2n+1}\in P_B^{\tau_1,\sigma_1}(x_{2n})
\quad\mbox{and}\quad
x_{2n+2}\in P_A^{\tau_2,\sigma_2}(x_{2n+1})
\quad
(n=0,1,\ldots),
\end{gather*}
where $\tau_1,\tau_2\in (0,1]$ and $\sigma_1,\sigma_2\in [0,1)$.
For instance, the projections on one of the sets, say, $A$ can be required to be exact, i.e., $\tau_2=1$ and $\sigma_2=0$.
Theorem~\ref{new_mai_res} remains applicable to this situation with $\tau:=\min\{\tau_1,\tau_2\}$ and $\sigma:=\max\{\sigma_1,\sigma_2\}$.
It is possible to obtain a sharper convergence estimate taking into account different ``inexactness'' parameters for each of the sets.
For that, one needs to amend the definition of alternating projections by considering the selection of the pair $\{x_{2n+1},x_{2n+2}\}$ as a single two-part iteration.
\end{remark}

\subsection{Convergence for inexact alternating projections under uniform regularity}

The motivation for the discussed below version of inexact projections comes from \cite[Section~6]{LewLukMal09}.

Given a point $x$ and a set $A$ in a Hilbert space and a number $\sigma\in[0,1)$, the \emph{$\sigma$-projection} of $x$ on $A$ is defined as follows:
\begin{gather}\label{bc}
P_{A}^\sigma(x):= \left\{a\in A \mid d\left(x-a,N_A(a)\right)\le \sigma\|x-a\|\right\}.
\end{gather}
Observe that
\begin{gather*}
P_{A}^0(x)= \left\{a\in A \mid x-a\in N_A(a)\right\}\supset P_{A}(x)
\end{gather*}
and the inclusion can be strict even in finite dimensions.
Furthermore, for any $\sigma\in[0,1)$, $P_{A}^\sigma(x)$ can contain points lying arbitrarily far from $x$.

\begin{definition}[Inexact alternating projections]\label{D32}
Given a number $\sigma\in[0,1)$,
$\{x_n\}$ is a sequence of $\sigma$-alternating projections for $\{A,B\}$ if
\begin{gather}\notag
x_{2n+1}\in P_B^{\sigma}(x_{2n})
\quad\mbox{and}\quad
x_{2n+2}\in P_A^{\sigma}(x_{2n+1}),
\\\label{D32-1}
\|x_{n+2}-x_{n+1}\|\le\|x_{n+1}-x_{n}\|
\quad
(n=0,1,\ldots).
\end{gather}
\end{definition}

The role of the monotonicity condition \eqref{D32-1} in Definition~\ref{D32} is to compensate for the lack of good projection properties of the $\sigma$-projection operator \eqref{bc}.
In the case of standard alternating projections (cf. Definition~\ref{AP}), this condition is satisfied automatically.

\begin{theorem}[Convergence of inexact alternating projections under uniform regularity]\label{T31}
Suppose that $\{A,B\}$ is uniformly regular at $\bx$, $A$ is super-regular at $\bx$ and $\sigma\in[0,1)$ satisfies
$$c_0:=\hat{c}[A,B](\bx)(1-\sigma^2)+\sigma^2+2\sigma\sqrt{1-\sigma^2}+\sigma<1.$$
Then, for any $c\in(c_0,1)$,
any sequence $\{x_k\}$ of $\sigma$-alternating projections for $\{A,B\}$ with initial points $x_0$ and $x_1$ sufficiently close to $\bx$ converges to a point in $A\cap B$ with $R-$linear rate $\sqrt{c}$.
\end{theorem}
\begin{proof}
Let $c\in(c_0,1)$ and choose a $c_1>\hat{c}[A,B](\bx)$ and a $\gamma>0$ such that \begin{align}\label{er0}
c_1(1-\sigma^2)+\sigma^2+(2\sigma+\gamma)\sqrt{1-\sigma^2}+\sigma<c.
\end{align}
By Proposition~\ref{P3}(iii) and Definition~\ref{sup_reg}(ii), there exists a $\delta>0$ such that
\begin{gather}\label{er1}
-\langle u,v\rangle\le c_1\|u\|\,\|v\|,
\\\label{er2}
\langle u,x-a\rangle\le \gamma\|u\|\,\|x-a\|
\end{gather}
for all $x,a\in A\cap \B_{\de}(\bx)$, $b\in B\cap \B_{\de}(\bx)$, $u\in N_{A}(a)$ and $v\in N_{B}(b)$.

Let $a_1\in A\cap\B_{\delta}(\bx)$, $b\in P_B^{\sigma}(a_1)\cap\B_{\delta}(\bx)$ and $a_2\in P_A^{\sigma}(b)\cap\B_{\delta}(\bx)$.
We are going to show that
\begin{align}\label{er5}
\|a_2-b\|\le c\|b-a_1\|.
\end{align}

By definition \eqref{bc}, for any $\varepsilon\in(0,1-\sigma)$, there exist $u\in N_{A}(a_2)$ and $v\in N_{B}(b)$ such that
\begin{gather}\label{er3}
\|b-a_2-u\|\le(\sigma+\varepsilon)\|b-a_2\|
\quad\mbox{and}\quad
\|a_1-b-v\|\le(\sigma+\varepsilon)\|a_1-b\|.
\end{gather}
Additionally, one can ensure that
\begin{gather}\label{er4}
\|u\|\le\sqrt{1-(\sigma+\varepsilon)^2}\|b-a_2\|
\quad\mbox{and}\quad
\|v\|\le\sqrt{1-(\sigma+\varepsilon)^2}\|a_1-b\|.
\end{gather}
Indeed, take any $u\in N_{A}(a_2)$ satisfying the first inequality in \eqref{er3}.
If $u=0$, the first inequality in \eqref{er4} is satisfied too.
Suppose $u\ne0$ and consider $u_1:=\ang{b-a_2,u}\frac{u}{\|u\|^2}$ -- the projection of $b-a_2$ on $\R u$.
Then
\begin{gather}\label{er7}
\|b-a_2\|^2=\|u_1\|^2+\|b-a_2-u_1\|^2,
\\\label{er10}
\|b-a_2-u_1\|\le\|b-a_2-u\|\le(\sigma+\varepsilon)\|b-a_2\|
\end{gather}
and there exists a $t\in(0,1]$ such that $u_2:=tu_1$ satisfies
\begin{gather}\label{er9}
\|b-a_2-u_2\|=(\sigma+\varepsilon)\|b-a_2\|
\end{gather}
(thanks to the continuity of the function $t\mapsto\|b-a_2-tu_1\|$).
Hence, $u_2\in N_{A}(a_2)$, vector $u_1-u_2$ is a projection of $b-a_2-u_2$ on $\R u$, i.e.,
\begin{gather}\label{er8}
\|b-a_2-u_2\|^2=\|u_1-u_2\|^2+\|b-a_2-u_1\|^2,
\end{gather}
and, using \eqref{er7}, \eqref{er8}, \eqref{er9}, and \eqref{er10},
\begin{align*}
\|u_2\|&=\|u_1\|-\|u_1-u_2\|
\\&
=\sqrt{\|b-a_2\|^2-\|b-a_2-u_1\|^2}-\sqrt{(\sigma+\varepsilon)^2\|b-a_2\|^2-\|b-a_2-u_1\|^2}
\\&
=\frac{(1-(\sigma+\varepsilon)^2)\|b-a_2\|^2}{\sqrt{\|b-a_2\|^2-\|b-a_2-u_1\|^2}+\sqrt{(\sigma+\varepsilon)^2\|b-a_2\|^2-\|b-a_2-u_1\|^2}}
\\&
\le\frac{(1-(\sigma+\varepsilon)^2)\|b-a_2\|^2}{\sqrt{1-(\sigma+\varepsilon)^2}\|b-a_2\|}=\sqrt{1-(\sigma+\varepsilon)^2}\|b-a_2\|.
\end{align*}
Similarly, given any $v\in N_{B}(b)$ satisfying the second inequality in \eqref{er3}, one can find a $v_2\in N_{B}(b)$ satisfying this inequality and, additionally, the second inequality in \eqref{er4}.

Making use of \eqref{er1}, \eqref{er3} and \eqref{er4}, we get
\begin{align*}
-\ang{b-a_2,a_1-b}&=-\ang{u,v}-\ang{u,a_1-b-v}
\\&
-\ang{b-a_2-u,v}-\ang{b-a_2-u,a_1-b-v}
\\&\le
c_1\|u\|\,\|v\|+\|u\|\,\|a_1-b-v\|
\\&
+\|b-a_2-u\|\,\|v\|+\|b-a_2-u\|\,\|a_1-b-v\|
\\&\le
\Bigl(c_1(1-(\sigma+\varepsilon)^2)+2(\sigma+\varepsilon)(\sqrt{1-(\sigma+\varepsilon)^2})
\\&
+(\sigma+\varepsilon)^2\Bigr)\|b-a_2\|\,\|a_1-b\|.
\end{align*}
At the same time, making use of \eqref{er2} and the first inequalities in \eqref{er3} and \eqref{er4}, we have
\begin{align*}
\langle b-a_2,a_1-a_2\rangle&=\langle u,a_1-a_2\rangle+\langle b-a_2-u,a_1-a_2\rangle
\\&\le
\gamma\|u\|\,\|a_1-a_2\|+(\sigma+\varepsilon)\|b-a_2\|\,\|a_1-a_2\|
\\&\le
(\gamma\sqrt{1-(\sigma+\varepsilon)^2}+(\sigma+\varepsilon))\|b-a_2\|\,\|a_1-a_2\|,
\end{align*}
Adding the last two estimates
and passing to limit as $\varepsilon\downarrow0$, we obtain
\begin{align*}
\|b-a_2\|^2\le
\Bigl(c_1(1-\sigma^2)+\sigma^2+(2\sigma+\gamma)\sqrt{1-\sigma^2}+\sigma\Bigr)\|b-a_2\|\,\|a_1-b\|.
\end{align*}
Thanks to \eqref{er0}, this proves \eqref{er5}.

Now we show that a sequence $\{x_n\}$ of $\sigma$-alternating projections for $\{A,B\}$ remains in $\B_{\delta}(\bx)$ if $x_0,x_1\in\B_{\rho}(\bx)$ where $\rho:=
\frac{1-c}{5-c}\delta<\delta$.
Let $n\in\N$ and $x_k\in\B_{\delta}(\bx)$, $k=0,1,\ldots,2n$.
It follows from \eqref{er5} that
\begin{align}\label{er6}
\|x_{2k}-x_{2k-1}\|
\le c^{k}\|x_{1}-x_{0}\|\quad (k=1,2,\ldots, n),
\end{align}
and consequently, employing also \eqref{D32-1},
\begin{align*}
\|x_{2n+2}-x_{0}\|&\le\|x_{2n+2}-x_{2n+1}\|+\|x_{2n+1}-x_{0}\|
\le\|x_{2n+2}-x_{2n+1}\|
\\&
+\sum_{k=1}^{n}\left(\|x_{2k+1}-x_{2k}\|+\|x_{2k}-x_{2k-1}\|\right)+\|x_{1}-x_{0}\|
\\&\le2\sum_{k=1}^{n}\|x_{2k}-x_{2k-1}\|+2\|x_{1}-x_{0}\|
\\&
\le 2\sum_{k=0}^{n}c^k\|x_{1}-x_{0}\|\le\frac{2}{1-c}\|x_1-x_0\|.
\end{align*}
Thus,
\begin{align*}
\max\{\|x_{2n+2}-\bx\|&,\|x_{2n+1}-\bx\|\}
\le\max\{\|x_{2n+2}-x_0\|,\|x_{2n+1}-x_0\|\}
\\
&+\|x_0-\bx\|
\le\frac{2}{1-c}\|x_1-x_0\|+\|x_0-\bx\|
\\
&\le\frac{2}{1-c}\|x_1-\bx\|+\frac{3-c}{1-c}\|x_0-\bx\|<\frac{5-c}{1-c}\rho=\delta,
\end{align*}
i.e., $x_{2n+1},x_{2n+2}\in\B_{\delta}(\bx)$.

Thanks to \eqref{er6}, $\{x_k\}$ is a Cauchy sequence containing two subsequences belonging to closed subsets $A$ and $B$, respectively.
Hence, it converges to a point in $A\cap B$ with $R-$linear rate $\sqrt{c}$.
\end{proof}
\begin{remark}
1. When the ``inexactness'' parameter $\sigma$ is small (cf. definition \eqref{D32}), then the assumptions of Theorem~\ref{T31} are easily satisfied (as long as ${\hat{c}[A,B](\bx)<1}$ and condition \eqref{D32-1} holds) while the convergence rate is close to the one guaranteed by Theorem~\ref{LLM} and \cite[Theorem~6.1]{LewLukMal09}.

2. One can also consider a more advanced version of inexact alternating projections than the one given in Definition~\ref{D32}:
\begin{gather*}
x_{2n+1}\in P_B^{\sigma_1}(x_{2n})
\quad\mbox{and}\quad
x_{2n+2}\in P_A^{\sigma_2}(x_{2n+1}),
\quad
(n=0,1,\ldots).
\end{gather*}
where $\sigma_1,\sigma_2\in [0,1)$.
Theorem~\ref{T31} remains applicable to this situation with $\sigma:=\max\{\sigma_1,\sigma_2\}$ (cf. Remark~\ref{R31}.3).

3. Observe that, thanks to \eqref{er5}, for odd values of $n$, condition \eqref{D32-1} is improved in the proof of Theorem~\ref{T31}:
\begin{gather*}
\|x_{n+2}-x_{n+1}\|\le c\|x_{n+1}-x_{n}\|,
\end{gather*}
where $c<1$.
However, the assumption is still needed to ensure that $x_{n+2}$ is not too far from $\bx$ and uniform and super-regularity conditions are applicable..

4. Constant $c_1$ in \eqref{er1} is an upper estimate of the cosine of the angle $\varphi$ between vectors $u$ and $-v$ while $\sigma+\varepsilon$ in \eqref{er3} can be interpreted as an upper estimate of the sine of the angles $\psi_1$ and $\psi_2$ between vectors $b-a_2$ and $u$ and $a_1-b$ and $v$, respectively.
One can use standard trigonometric identities and inequalities \eqref{er1} and \eqref{er3} to obtain an upper estimate of the cosine of the angle $\varphi-\psi_1-\psi_2$ between vectors $b-a_2$ and $b-a_1$ and possibly improve the convergence estimate in the statement of Theorem~\ref{T31}.

5. If both subsets $A$ and $B$ are super-regular, then
in the proof of Theorem~\ref{T31}, one can establish an analogue of \eqref{er5} with subsets $A$ and $B$ interchanged:
\begin{align*}
\|b_2-a\|\le c\|a-b_1\|,
\end{align*}
where $b_1\in B\cap\B_{\delta}(\bx)$, $a\in P_A^{\sigma}(b_1)\cap\B_{\delta}(\bx)$ and $b_2\in P_B^{\sigma}(a)\cap\B_{\delta}(\bx)$.
This guarantees an improvement with rate $c$ on each iteration.
As a result, one obtains a better overall  $R-$linear rate $c$.

6. The conclusion of Theorem~\ref{T31} remains true if one replaces the assumptions of uniform regularity of $\{A,B\}$ (and the regularity constant $\hat{c}[A,B](\bx)$) and super-regularity of $A$ with BLPW-restricted regularity (and the regularity constant $\hat{c}_1[A,B](\bx)$) and $B$-super-regularity, respectively, accompanied by appropriate adjustments in the definition of $\sigma$-projections.
\end{remark}
\section*{Acknowledgements}

The authors wish to express their gratitude to the anonymous referee for the careful reading of the manuscript and helpful comments and suggestions.

\bibliographystyle{acm}
\bibliography{BUCH-Kr,Kruger,KR-tmp}
\end{document}